\documentclass[10pt, reqno]{amsart}
\usepackage{amsmath}
\usepackage{cases}
\usepackage{mathrsfs}
\usepackage{cite}
\usepackage{bbm}
\usepackage{amssymb}
\usepackage{amscd}
\usepackage{amsfonts,latexsym,amsmath,amsthm,amsxtra,mathdots,amssymb,latexsym,mathabx}
\usepackage[all,cmtip]{xy}
\RequirePackage{amsmath} \RequirePackage{amssymb}
\usepackage{color}
\usepackage{colordvi}
\usepackage{multicol}
\usepackage{hyperref}
\usepackage{mathtools}
\usepackage[margin=1.1in]{geometry}
\usepackage{xcolor}
\hypersetup{
    colorlinks,
    linkcolor={red!50!black},
    citecolor={blue!50!black},
    urlcolor={blue!80!black}
}

     \newcommand{\BN}{{\mathbb {N}}}
     
     \newcommand{\BR}{{\mathbb {R}}}
     \newcommand{\BT}{{\mathbb {T}}}

     \newcommand{\BZ}{{\mathbb {Z}}}

     \newcommand{\CJ}{{\mathcal {J}}}

    \newcommand{\frI}{{\mathfrak{I}}}
    \newcommand{\frC}{{\mathfrak{C}}}
    \newcommand{\frK}{{\mathfrak{K}}}

\def\-{^{-1}}
\def\sumx{\ \sideset{}{^\star}\sum}

\newtheorem{Theorem}{Theorem}[section]
\newtheorem{Lemma}[Theorem]{Lemma}

\newtheorem{Proposition}[Theorem]{Proposition}
\newtheorem{Remark}[Theorem]{Remark}

\begin{document}

\title{A new subconvex bound for $\rm GL(3)$ $L$-functions in the $t$-aspect}
 \author{Keshav Aggarwal}
\date{}
\address{Department of Mathematics and Statistics, University of Maine\\ 334 Neville Hall\\
Orono, Maine 04401, USA}
\email{keshav.aggarwal@maine.edu}

\begin{abstract}
We revisit Munshi's proof of the $t$-aspect subconvex bound for $\rm GL(3)$ $L$-functions, and we are able to remove the `conductor lowering' trick. This simplification along with a more careful stationary phase analysis allows us to improve Munshi's bound to,
$$ L(1/2+it, \pi) \ll_{\pi, \epsilon} (1+|t|)^{3/4-3/40+\epsilon}. $$
\end{abstract}
\maketitle

\section{Introduction and statement of result}
Let $\pi$ be a Hecke cusp form of type $(\nu_1,\nu_2)$ for $\rm SL(3,\BZ)$. Let the normalized Fourier coefficients be given by $\lambda(m_1,m_2)$ (so that $\lambda(1,1)=1$). The $L$-series associated with $\pi$ is given by
$$L(s,\pi) = \sum_{n\geq1}\lambda(1,n)n^{-s}, \quad \text{ for } Re(s)>1.$$

In this paper, we follow Munshi \cite{Mun4} but remove the `conductor lowering' trick. This simplification makes the stationary phase analysis cleaner, which improves certain estimates and allows us to improve the subconvex bound exponent. We prove the following theorem.

\begin{Theorem}\label{main theorem gl3} 
Let $\pi$ be a Hecke-Maass cusp form for $SL(3,\BZ)$. Then for any $\epsilon>0$
\begin{equation*}
L\left(1/2+it,\pi\right)\ll_{\pi, \epsilon} t^{3/4-3/40+\epsilon}.
\end{equation*}
\end{Theorem}

We should mention that one can modify our approach in \cite{Agg2018published} for a $t$-aspect subconvex bound for $\rm GL(2)$ $L$-functions to remove the `conductor lowering' trick. Indeed, it suffices to replace the delta method used in \cite{Agg2018published} by an averaged version of the following uniform partition of the circle to detect $n=0$
\begin{equation*}
\delta(n=0) = \frac{1}{q}\sum_{a\bmod q}e\bigg(\frac{an}{q}\bigg)\int_0^1 e\bigg(\frac{nx}{q}\bigg) dx, \quad \text{ for } q\in\BN.
\end{equation*}
By choosing $q$ to be of size approximately $t^{1/3}$, one can get the Weyl bound for $\rm GL(2)$ $L$-functions in $t$-aspect. This circle method seems insufficient to obtain a $\rm GL(3)$ $t$-aspect subconvex bound, so we use Kloosterman's version of the circle method (Lemma \ref{circle method}).

A $t$-aspect bound for self-dual $\rm GL(3)$ $L$-functions was first established by Li \cite{Li1}, and improved upon by McKee, Sun and Ye \cite{McKee-Sun-Ye}, Sun and Ye \cite{Sun-Ye-2019} and Nunes \cite{Nunes2017}. A $t$-aspect bound for a general $\rm SL(3,\BZ)$ $L$-function was proved by Munshi by a completely different approach. We revisit Munshi's proof and improve upon his result. Although the bound obtained here is weaker than that in \cite{McKee-Sun-Ye, Sun-Ye-2019, Nunes2017}, it holds for any Hecke-Maass cusp form for $\rm SL(3,\BZ)$, and not just the self-dual forms. We note that we get the same exponent as Sun and Zhao \cite{SunZhaoDepthAspect}, whose work is on bounding twists of $\rm GL(3)$ $L$-functions in depth aspect. They use Kloosterman's version of circle method, along with a `conductor lowering' trick appropriate for the depth aspect. On the other hand, our result in $t$-aspect doesn't need a `conductor lowering' trick. Other works that deal with the subconvex bound problem for degree three $L$-functions include \cite{Blo12, Mun1, Mun2, mun3, Mun4, Munshicrelle, Ho-Ne2018, Lin, Sun-Ye-2019}.

We start with applying the approximate functional equation (Lemma \ref{AFE})
\begin{equation*}
L(1/2+it, \pi) \ll \sup_{1\leq N\ll t^{3/2+\epsilon}} \frac{S(N)}{N^{1/2}},
\end{equation*}
where 
\begin{equation}\label{main object gl3}
\begin{split}
S(N)=\sum_{n=1}^{\infty}\lambda(1,n)\,n^{-it}V\left(\frac{n}{N}\right).
\end{split}\end{equation}
$V$ is a smooth function with bounded derivatives and supported in $[1,2]$, that is allowed to depend on $t$. An application of Cauchy-Schwarz inequality applied to $S(N)$ followed by the Ramanujan bound on average (Lemma \ref{Ram bound}) gives the trivial bound $S(N)\ll N^{1+\epsilon}$. Therefore it suffices to beat the trivial bound $O(N^{1+\epsilon})$ of $S(N)$ for $N$ in the range $t^{3/2-\delta}<N<t^{3/2+\epsilon}$ and some $\delta>0$.

The next step is to separate the oscillations by writing
\begin{equation*}
S(N) = \sum_{n=1}^\infty \sum_{r=1}^\infty \lambda(1,n)r^{-it}U\bigg(\frac{n}{N}\bigg)V\bigg(\frac{r}{N}\bigg)\delta(n-r=0),
\end{equation*}
where $U$ is a smooth function compactly supported on $[1/2, 5/2]$ and $U(x)=1$ for $x\in [1,2]$. We note that $U(x)V(x)=V(x)$. The separation of oscillations is done by using Kloosterman's version of the circle method.

\begin{Lemma}[Kloosterman circle method]\label{circle method} Let $Q\geq1$ and $n\in\BZ$. Then
\begin{equation*}
\delta(n=0) = 2Re\int_0^1\underset{1\leq q\leq Q<a\leq q+Q}{\sum\sumx} \frac{1}{aq}e\bigg(\frac{n\overline{a}}{q}-\frac{nx}{aq}\bigg)dx,
\end{equation*}
where $e(x)=e^{2\pi ix}$ and $\star$ on the summation denotes the coprimality condition $(a,q)=1$.
\end{Lemma}
\begin{proof}
See \cite[Section 20.3]{Iw-Ko}.
\end{proof}

The next steps are to apply dual summation formulas to the $n$ and the $r$ sums, followed by stationary phase analysis of the oscillatory integrals, and a final application of Cauchy-Schwarz and Poisson summation to the $n$-sum. We give detailed heuristics of these calculations in Section \ref{proof sketch}. We prove the following main proposition.

\begin{Proposition}\label{main prop gl3}
Let $S(N)$ be given by equation \eqref{main object gl3} and $t^{1+\epsilon}<N\ll t^{3/2+\epsilon}$. Let $Q$ be a parameter satisfying $N^{1+\epsilon}/t<Q<N^{1/2}$. For any $\varepsilon>0$, we have
\begin{align*}
S(N) \ll_\varepsilon
  \frac{N^{3/2}t^\varepsilon}{Q^{3/2}} + QN^{1/4}t^{1/2+\varepsilon}  \ \ \ \textrm{if} \ \ \ \ t^{1+\epsilon}\ll N \ll t^{3/2+\epsilon}.
\end{align*}
\end{Proposition}

By choosing $Q=N^{1/2}/t^{1/5}$, we obtain the bound of $S(N)\ll_\varepsilon N^{3/4}t^{3/10+\varepsilon}$. Comparing this with the trivial bound of $S(N)\ll N^{1+\varepsilon}$, the optimal range of $N$ where the above proposition gives a better than the trivial estimate is $t^{6/5}<N<t^{3/2+\epsilon}$. For $N\ll t^{6/5}$, we use the trivial bound $S(N)\ll N^{1+\varepsilon}$. Taking the supremum of $\frac{S(N)}{N^{1/2}}$ over the range $N\ll t^{3/2+\epsilon}$, the proposition implies Theorem \ref{main theorem gl3}.

\subsection*{Notations} In the rest of the paper, we use the notation $e(x)=e^{2\pi ix}$. Let $a$ and $b$ be two positive real numbers. We denote $a\sim b$ to mean $k_1<a/b<k_2$ for some absolute constants $k_1,k_2>0$. We use $\epsilon$ to denote an arbitrarily small positive constant that can change depending on the context. We denote $a\asymp b$ to mean $t^{-\epsilon}b\ll_\epsilon a\ll_\epsilon t^\epsilon b$ asymptotically as $|t|\rightarrow \infty$. We must add that for brevity of notation, we assume $t>0$. Indeed, the same analysis holds for $t<0$ by replacing $t$ with $-t$ appropriately. 

\section{Outline of the Proof}\label{proof sketch}
In this section, we give detailed heuristics of the proof. We use Kloosterman's version of the circle method to separate the $n$ and the $r$ variables in $S(N)$. Roughly
\begin{equation}\label{start sketch}
S(N) \asymp \int_0^1\underset{1\leq q\leq Q<a\leq Q}{\sum\sumx} \frac{1}{aq} \sum_{r\asymp N}r^{-it} e(r\overline{a}/q) e(-rx/aq)\sum_{n\asymp N} A(n) e(-n\overline{a}/q)e(-nx/aq)dx,
\end{equation}
where $A(n)=\lambda(1,n)$. Trivial bound gives $S(N)\ll N^2$. So we need to save $N$ and a bit more. We start by applying dual summation formulas to the $r$-sum and the $n$-sum.
\subsubsection{Poisson to the $r$-sum} The conductor of the $r$-sum is $q(t+N/aq)$. If $Q\ll \sqrt{N/t}$, then $t<N/aq$, so that $q(t+N/aq)\asymp N/a$. The dual length after Poisson summation is then $1/Q\ll 1$. If $Q$ has size, that is if $Q>t^\epsilon$, only $r=0$ contributes. In this case the congruence condition $\delta(\overline{a}\equiv -r\bmod q)$ implies $q=1$ and $a=\lfloor{Q+1}\rfloor$. That is, both the $q$ and $a$ sums vanish, and all the contributions of additive twists due to the circle method become trivial. We are therefore not able to improve upon the convexity bound. We choose $Q$ so that $q(t+N/aq)\asymp qt$ for some $q\leq Q$. Let $U$ be a smooth bump function controlling $r\sim N$. Poisson summation transforms the $r$-sum as
\begin{equation}\label{Poisson sketch}
\sum_{r\geq1} e(r\overline{a}/q)r^{-it}e(-rx/aq)U(r/N) \longleftrightarrow N^{1-it}\sum_{|r|\ll qt/N} \delta(\overline{a}\equiv -r\bmod q)U^\dagger\bigg(\frac{N(ra-x)}{aq},1-it\bigg).    
\end{equation}
Thus $a$ is determined $\bmod\ q$ and $a\sim Q$. $U^\dagger(v,w)$ is a highly oscillatory integral that is negligible unless $|v|\asymp|w|$. We also observe that only $r=0$ exists if $Q<N/t^{1-\epsilon}$. So we choose $Q>N/t^{1-\epsilon}$.

\subsubsection{Voronoi to the $n$-sum} The conductor of the $n$-sum is $q^3(N/aq)^3$, so the new length after Voronoi formula is $N^2/Q^3$. Let $V$ be a smooth bump function controlling $n\sim N$. Voronoi summation transforms the $n$-sum as
\begin{equation}\label{Voronoi sketch}
\begin{split}
\sum_{n\geq1}A(n)e(nr/q)e(nx/aq)V(n/N) \longleftrightarrow N^{1/2}\sum_{n\ll \frac{N^2}{Q^3}} \frac{A^\star(n)}{\sqrt{n}}\frac{S(\overline{r},n;q)}{\sqrt{q}}&\int_{-N/qQ}^{N/qQ} \bigg(\frac{nN}{q^3}\bigg)^{-i\tau} \gamma(-1/2+i\tau)\\&\times V^\dagger(Nx/aq,1/2+i\tau)d\tau.
\end{split}
\end{equation}
where we have assumed $Q^2<N$ and $A^\star(n)=\lambda(n,1)$. Combining \eqref{Poisson sketch} and \eqref{Voronoi sketch}, \eqref{start sketch} transforms into
\begin{equation}\label{after dual sum sketch}
S(N)\asymp N^{3/2}\sum_{n\ll\frac{N^2}{Q^3}}\frac{A^\star(n)}{\sqrt{n}}\sum_{1\leq q\leq Q}\frac{1}{aq}\sum_{|r|\ll \frac{qt}{N}}\frac{S(\overline{r},n;q)}{\sqrt{q}} I(n,r,q),
\end{equation}
where $I(n,r,q)$ is a highly oscillatory double integral over $x,\tau$. 

\subsubsection{Analysis of the integrals} A major part of the paper is performing a robust stationary phase analysis of $I(n,r,q)$. We expect a square root saving of `the size of the oscillation'. $U^\dagger$ in \eqref{Poisson sketch} has oscillation of size $t$. Thus
\begin{equation}\label{Udagger sketch}
U^\dagger\bigg(\frac{N(ra-x)}{aq},1-it\bigg) \asymp \frac{1}{\sqrt{t}}\bigg(\frac{N(ra-x)}{aqt}\bigg)^{it}\times\text{smooth fn}.
\end{equation}
Similarly, $V^\dagger$ in \eqref{Voronoi sketch} has oscillation of size $\tau\asymp Nx/Qq$. Thus
\begin{equation}\label{Vdagger sketch}
V^\dagger(Nx/aq, 1/2+i\tau)\asymp \frac{1}{\sqrt{\tau}}\bigg(\frac{Nx}{aq\tau}\bigg)^{i\tau}\times\text{smooth fn}.
\end{equation}
From \eqref{Udagger sketch} and \eqref{Vdagger sketch}, the $x$-integral in $I(n,r,q)$ is therefore
\begin{equation*}
\int_0^1 (ra-x)^{it}x^{i\tau}\times\text{smooth fn} \asymp \frac{1}{\sqrt{N/Qq}}\times\tau\text{-oscillations},
\end{equation*}
where the above asymptotic comes from the observation that the $x$-oscillation is of size $\tau$, which is at most $N/Qq$. The biggest contribution comes from the range $\tau\asymp N/Qq$. Therefore \eqref{after dual sum sketch} is
\begin{equation}\label{after sp sketch}
S(N)\asymp \frac{N^{1/2}}{t^{1/2}} \sum_{n\ll\frac{N^2}{Q^3}}\frac{A^\star(n)}{\sqrt{n}}\sum_{1\leq q\leq Q}\sum_{|r|\ll \frac{qt}{N}} \frac{S(\overline{r},n;q)}{\sqrt{q}} \int_{|\tau|\asymp N/qQ} g(q,r,\tau) n^{-i\tau} d\tau,
\end{equation}
where $g(q,r,\tau)=O(1)$ and has $\tau$-oscillations. Recalling the restriction $Q^2<N$, the above gives the bound $S(N)\ll \frac{N^{3/2}t^{1/2}}{Q^{3/2}}\asymp N^{3/4}t^{1/2} $. Since $N\ll t^{3/2}$, this gives us the bound of $L(1/2+it,\pi)\ll t^{7/8+\epsilon}$. We therefore need to save $t^{1/8}$ and a bit more.

\subsubsection{Cauchy-Schwarz and Poisson to the $n$-sum} The biggest contribution in \eqref{after sp sketch} comes from the largest $n$, that is $n\asymp N^2/Q^3$. Applying Cauchy-Schwarz inequality with the $n$-sum outside, followed by the Ramanujan bound on average (Lemma \ref{Ram bound})
\begin{equation*}
S(N)\ll \frac{N^{1/2}}{t^{1/2}}\bigg(\sum_{n\asymp\frac{N^2}{Q^3}}\bigg|\sum_{1\leq q\leq Q}\sum_{|r|\ll \frac{qt}{N}} \frac{S(\overline{r},n;q)}{\sqrt{q}} \int_{-N/qQ}^{N/qQ} g(q,r,\tau) n^{-i\tau} d\tau\bigg|^2\bigg)^{1/2}.
\end{equation*}
Opening the absolute value squared, the $n$-sum inside the square root is
$$\sum_{n\asymp\frac{N^2}{Q^3}}\frac{S(\overline{r}_1,n;q_1)}{\sqrt{q_1}} \frac{S(\overline{r}_2,n;q_2)}{\sqrt{q_2}}n^{-i(\tau_2-\tau_1)}U\bigg(\frac{n}{N^2/Q^3}\bigg). $$
Application of Poisson summation to the $n$-sum gives the new length, $\frac{q^2(N/qQ)}{N^2/Q^3}\asymp Q^2q/N$, where $q_1,q_2\sim q$. The above sum transforms into
\begin{equation}\label{final poisson sketch}
\frac{N^2}{Q^3}\sum_{|n|\ll \frac{Q^2q}{N}}\frac{1}{q_1q_2}\frC U^\dagger\bigg(\frac{N^2n}{q_1q_2Q^3}, 1-i(\tau_1-\tau_2)\bigg),
\end{equation}
where $\frC$ is a character sum. When $n=0$, $\frC$ vanishes unless $q_1=q_2$ and $r_1=r_2$, in which case $\frC\ll q_1q_2$. Moreover, $U^\dagger$ is negligible unless $\tau_1\asymp\tau_2$. In that case, $U^\dagger(0,\star)\ll 1$. Therefore the contribution of $n=0$ towards $S(N)$ is
\begin{equation}\label{Sflat sketch}
S^\flat(N)\ll \frac{N^{1/2}}{t^{1/2}}\bigg(\sum_{1\leq q\leq Q}\sum_{|r|\ll qt/N} \frac{|\frC|}{q^2} \frac{N^2}{Q^3}\int_{-N/Qq}^{N/Qq}|g(q,r,\tau)|^2 d\tau \bigg)^{1/2}\ll \frac{N^{3/2}}{Q^{3/2}}.
\end{equation}
When $n\neq0$, $\frC\ll \sqrt{q_1q_2}$, and the $U^\dagger$ as given in \eqref{final poisson sketch} is asymptotic to
\begin{equation*}
U^\dagger\bigg(\frac{N^2n}{q_1q_2Q^3}, 1-i(\tau_1-\tau_2)\bigg)\asymp \bigg(\frac{q_1q_2Q^3}{N^2n}\bigg)^{1/2}h(\tau_1,\tau_2),
\end{equation*}
where $h(\tau_1,\tau_2)$ is an oscillatory function of size $1$. We again expect square root cancellation in the $\tau_1,\tau_2$-integral. That is
$$\underset{|\tau_i|\asymp N/Qq_i}{\int\int}g(q_1,r_1,\tau_1)\overline{g(q_2,r_2,\tau_2)}h(\tau_1,\tau_2)d\tau_1\ d\tau_2\ll \frac{N}{Q\sqrt{q_1q_2}}. $$
Therefore the contribution of the $n\neq0$ terms towards $S(N)$ is
\begin{equation*}
S^\sharp(N)\ll \frac{N^{1/2}}{t^{1/2}}\bigg(\underset{1\leq q_1,q_2\leq Q}{\sum\sum}\  \underset{|r_i|\ll q_it/N}{\sum\sum} \frac{|\frC|}{q_1q_2}\frac{N^2}{Q^3}\frac{N}{Qq}\sum_{0\neq|n|\ll Q^2q/N}\frac{Q^{3/2}\sqrt{q_1q_2}}{n^{1/2}}\bigg)^{1/2}\ll QN^{1/4}t^{1/2}.
\end{equation*}
Matching up the contributions of $S^\flat(N)$ and $S^\sharp(N)$, the optimal choice of $Q$ is $Q=N^{1/2}/t^{1/5}$. In this case, we get $S(N)\ll N^{3/4}t^{3/10+\epsilon}$. Since $N\ll t^{3/2+\epsilon}$, we get the bound $L(1/2+it,\pi)\ll_{\pi,\epsilon} (1+|t|)^{3/4-3/40+\epsilon}$.

We compare this with Munshi's heuristics. Munshi chooses $Q=\sqrt{N/K}$ with $t^{1/4}<K<t^{1/2}$ for some parameter $K$. Since we choose $Q=N^{1/2}/t^{1/5}$, our choice of $Q$ satisfies Munshi's restriction. We should also mention that our improvement comes in the step of stationary phase analysis of $I(n,r,q)$, which allows us to improve the bound for $S^\flat(N)$ as given in \eqref{Sflat sketch}. This is mentioned in Remark \ref{main_remark}. Indeed, our bound for $S^\sharp(N)$ matches that of Munshi.

Moreover, we observe that one needs $Q$ to be bigger than $N/t$ after Poisson summation, and smaller than $N^{1/2}$ after Voronoi summation for this approach to work. In the $\rm GL(n)$ $t$-aspect problem, say we choose $Q=N^{1/2}/t^{\alpha_n}$ for some $\alpha_n>0$. The restriction $N/t<Q<N^{1/2}/t^{\alpha_n}$ implies our analysis is valid in the range $N<t^{2-2\alpha_n}$. Since $N\ll t^{n/2}$ we are forced to choose $\alpha_n\leq0$ for $n\geq4$. Therefore this approach doesn't give a subconvex bound result for $n\geq4$.

\section{Preliminaries on automorphic forms}

Let $\pi$ be a Maass form of type $(\nu_1,\nu_2)$ for $\rm SL(3,\BZ)$, which is an eigenfunction for all the Hecke operators. Let the Fourier coefficients be $\lambda(n_1,n_2)$, normalized so that $\lambda(1,1)=1$. The Langlands parameter $(\alpha_1,\alpha_2, \alpha_3)$ associated with $\pi$ are $\alpha_1=-\nu_1-2\nu_2+1,\ \alpha_2= -\nu_1+\nu_2,\ \alpha_3= 2\nu_1+\nu_2-1$. The dual cusp form $\tilde{\pi}$ has Langlands parameters $(-\alpha_3, -\alpha_2, -\alpha_1)$. $L(s,\pi)$ satisfies a functional equation 
$$\gamma(s,\pi)L(s,\pi) = \gamma(s,\tilde{\pi})L(1-s,\tilde{\pi}), $$
where $\gamma(s,\pi)$ and $\gamma(s,\tilde{\pi})$ are the associated gamma factors. We refer the reader to Goldfeld's book on automorphic forms for $\rm GL(n)$ \cite{goldfeldbook} for the theory of automorphic forms on higher rank groups.

\subsection{Approximate functional equation and Voronoi summation formula}
We are interested in bounding $L(s,\pi)$ on the critical line, $Re(s)=1/2$. For that, we approximate $L(1/2+it, \pi)$ by a smoothed sum of length $t^{3/2+\epsilon}$. This is known as the approximate functional equation and is proved by applying Mellin transformation to $f$ followed by using the above functional equation.

\begin{Lemma}[{\cite[Theorem 5.3]{Iw-Ko}}]\label{AFE}
Let $G(u)$ be an even, holomorphic function bounded in the strip $-4\leq Re(u)\leq 4$ and normalized by $G(0)=1$. Then for $s$ in the strip $0\leq \sigma\leq 1$ 
$$ L(s,\pi) = \sum_{n\geq1} \lambda(1,n)n^{-s}V_s(n) +  \frac{\gamma(s,\tilde{\pi})}{\gamma(s,\pi)} \sum_{n\geq1} \overline{\lambda(1,n)}n^{-(1-s)}V_{1-s}(n) - R,$$
where $R\ll 1$,
$$V_s(y) = \frac{1}{2\pi i}\int_{(3)}y^{-u}G(u)\frac{\gamma(s+u,\tilde{\pi})}{\gamma(s,\pi)}\frac{du}{u}, $$
and $\gamma(s,\pi)$ is a product of certain $\Gamma$-functions appearing in the functional equation of $L(s,f)$.
\end{Lemma}
\begin{Remark}
On the critical line, Stirling's approximation to $\gamma(s,\pi)$ followed by integration by parts to the integral representation of $V_{1/2\pm it}(n)$ gives arbitrary saving for $n\gg (1+|t|)^{3/2+\epsilon}$.
\end{Remark}

One of the main tools in our proof is a Voronoi type formula for $\rm GL(3)$. Let $h$ be a compactly supported smooth function on  $(0, \infty)$, and let $\tilde{h}(s)=\int_0^\infty h(x)x^{s-1}\ dx$ be its Mellin transform. For $\sigma>-1+\max\{-Re(\alpha_1), -Re(\alpha_2), -Re(\alpha_3)\}$ and $\ell=0, 1$, define
\begin{equation*}
\gamma_\ell(s) = \frac{\pi^{-3s-3/2}}{2}\prod_{i=1}^3\frac{\Gamma(\frac{1+s+\alpha_i+\ell}{2})}{\Gamma(\frac{-s-\alpha_i+\ell}{2})}.
\end{equation*}
Further set $\gamma_\pm(s) = \gamma_0(s)\mp i\gamma_1(s)$ and let
\begin{equation*}
H_\pm(y) = \frac{1}{2\pi i}\int_{(\sigma)} y^{-s}\gamma_\pm(s)\tilde{h}(-s)\ ds. 
\end{equation*}
We need the following Voronoi type formula (See \cite{Li1, Miller-Schmid}).

\begin{Lemma}\label{gl3voronoi}
Let $h$ be a compactly supported smooth function on $(0, \infty)$. We have
\begin{equation*}
\sum_{n=1}^\infty \lambda(1,n) e(an/q)h(n) = q\sum_\pm \sum_{n_1|q}\sum_{n_2=1}^\infty \frac{\lambda(n_2,n_1)}{n_1n_2} S(\overline{a}, \pm n_2; q/n_1) H_\pm (n_1^2n_2/q^3),
\end{equation*}
where $(a,q)=1$ and $a\overline{a}\equiv 1\bmod q$.
\end{Lemma}

Stirling approximation of $\gamma_{\pm}(s)$ gives $\gamma_\pm(\sigma+i\tau)\ll 1+|\tau|^{3\sigma+3/2}$. Moreover on $Re(s)=-1/2$
\begin{equation}\label{stirling gl3}
\gamma_\pm(-1/2+i\tau) = (|\tau|/e\pi)^{3i\tau}\Phi_\pm(\tau), \qquad \text{ where } \Phi_\pm'(\tau)\ll |\tau|^{-1}.
\end{equation}

We will also use the Ramanujan bound on average which follows from the Rankin-Selberg theory.
\begin{Lemma}[Ramanujan bound on average] \label{Ram bound}
We have
\begin{equation*}
    \underset{n_1^2n_2\leq x}{\sum\sum}|\lambda(n_2,n_1)|^2 \ll_{\pi,\epsilon} x^{1+\epsilon}.
\end{equation*}
\end{Lemma}

\section{Stationary phase analysis}

We need to use stationary phase analysis for oscillatory integrals. Let $\mathfrak{I}$ be an integral of the form
\begin{equation}\label{eintegral}
\mathfrak{I} = \int_a^b g(x)e(f(x))dx,
\end{equation}
where $f$ and $g$ are  real valued smooth functions on  $\BR$. The fundamental estimate for integrals of the form \eqref{eintegral} is the $r^{th}$-derivative test 
\begin{equation}\label{rth der test}
\frI\ll \bigg(\underset{[a,b]}{Var}\, g(x)\bigg)\bigg/\bigg(\min_{[a,b]}|f^{(r)}(x)|^{1/r}\bigg).
\end{equation}
We will however need sharper estimates and will use the stationary phase analysis as given by Huxley \cite{HUX} to analyze $\frI$. Moreover, in the case when the stationary point lies far enough from the interval $[a,b]$, we use Lemma 8.1 of Blomer, Khan and Young \cite{bky} on stationary phase analysis to show that $\frI$ is arbitrarily small. For completeness, we state the results here.

The following estimate are in terms of the parameters $\Theta_f, \Omega_f$ and $\Omega_g$ for which the derivatives satisfy
\begin{equation}\label{assertion 1}
f^{(i)}(x)\ll \frac{\Theta_f}{\Omega_f^i}, \qquad g^{(j)}(x)\ll \frac{1}{\Omega_g^j}.
\end{equation}

For the second assertion of the following lemma, we also require
\begin{equation}\label{assertion 2}
    f''(x) \gg \frac{\Theta_f}{\Omega_f^2}.
\end{equation}

\begin{Lemma}\label{huxley}
Suppose $f$ and $g$ are real valued smooth functions satisfying \eqref{assertion 1} for $i=2, 3$ and $j=0, 1, 2$. Let $\Omega_f\gg (b-a)$.
\begin{enumerate}
\item Suppose $f^\prime$ and $f^{\prime \prime}$ do not vanish   on the interval $[a, b]$. Let $\Lambda = \min_{ x\in [a, b]} |f^\prime (x)| $. Then we have
\[
\mathfrak{I} = \frac{g(b)e(f(b))}{2\pi if'(b)} - \frac{g(a)e(f(a))}{2\pi if'(a)} + O\bigg(\frac{\Theta_f}{ \Omega_f^2 \Lambda^3} \bigg( 1 +\frac{\Omega_f}{\Omega_g} +\frac{\Omega_f^2}{\Omega_g^2} \frac{\Lambda}{\Theta_f/ \Omega_f} \bigg)\bigg).
\]

\item Suppose that $f^\prime(x)$ changes sign from negative to positive at $x = x_0$ with $a<x_0 <b$. Let $\kappa= \min \{  b-x_0, x_0-a \}$. Further suppose that bound in equation \eqref{assertion 1} holds for $i=4$ and \eqref{assertion 2} holds. Then we have the following  asymptotic expansion
\[
\mathfrak{I} = \frac{g(x_0) e( f(x_0) + 1/8)}{\sqrt{f^{\prime \prime } (x_0)}} + \frac{g(b)e(f(b))}{2\pi if'(b)} - \frac{g(a)e(f(a))}{2\pi if'(a)} + O\bigg( \frac{\Omega_f^4}{ \Theta_f^2 \kappa^3} +  \frac{\Omega_f}{ \Theta_f^{3/2}  } +  \frac{\Omega_f^3}{ \Theta_f^{3/2} \Omega_g^2 }\bigg). 
\]
\end{enumerate}
\end{Lemma}


The above result is presented in Theorem 1 and Theorem 2 of \cite{HUX}. We shall also use the following lemma from \cite{bky} when the unique point $x_0$ lies away from $(a,b)$.

\begin{Lemma} \cite[Lemma 8.1]{bky} \label{bky no sp}
Let $\Theta_f\geq1$, $\Omega_f, \Omega_g, \Lambda>0$, and suppose that $f$ and $g$ are smooth real valued functions on the interval $[a, b]$, with $g$ supported on $[a,b]$. Let $f$ and $g$ satisfy \eqref{assertion 1} for $i\geq 2$ and $j\geq0$. Moreover, let $\Lambda = \min_{x\in[a,b]}|f'(x)|$.  Then $\frI$ satisfies
$$\frI\ll_A |b-a|[(\Omega_f\Lambda/\sqrt{\Theta_f})^{-A}+(\Lambda\Omega_g)^{-A}].$$
\end{Lemma}

We shall also use the following estimates on exponential integrals in two variables. Let $f(x, y)$ and $g(x, y)$  be two real valued smooth functions on the rectangle $[a, b] \times [c,d]$.  We consider the exponential integral in two variables given by
\begin{equation*}
\int_a^b \int_c^d g(x, y) e(f(x, y)) dx  \ dy.
\end{equation*}
Suppose there exist parameters $p_1, p_2>0$ such that

\begin{align} \label{conditionf}
\frac{\partial^2 f}{\partial^2 x}\gg p_1^2, \hspace{1cm} \frac{\partial^2 f}{\partial^2 y}\gg p_2^2,\hspace{1cm}    \frac{\partial^2 f(x, y)}{\partial^2 x} \frac{\partial^2 f}{\partial^2 y} -  \left[\frac{\partial^2 f}{\partial x \partial y} \right]^2 \gg p_1^2 p_2^2,  
\end{align}  for all $x, y \in [a, b] \times [c,d]$. Then we have (See \cite[Lemma 4]{BR2})

\[
 \int_a^b \int_c^d  e(f(x, y)) dx dy \ll \frac{1}{p_1 p_2}. 
\] Further suppose that $ \textrm{Supp}(g) \subset (a,b) \times (c,d)$. The total variation of $g$ equals

\begin{equation*}
\textrm{var}(g) = \int_a^b \int_c^d  \left|  \frac{\partial^2 g(x, y)}{\partial x \partial y} \right| dx dy.
\end{equation*}
 We have the following result (see \cite[Lemma 5]{BR2}).
 \begin{Lemma} \label{double expo sum}
 Let $f$ and  $g$ be as above. Let $f$ satisfies the conditions given in  equation \eqref{conditionf}. Then we have
 \[
 \int_a^b \int_c^d g(x, y) e(f(x, y)) dx dy \ll \frac{\textrm{var}(g)}{p_1 p_2},
 \] with an absolute implied constant. 
 \end{Lemma}

\subsection{A Fourier-Mellin transform} Let $U$ be a smooth real valued function supported on the interval $[a, b] \subset (0, \infty)$ and satisfying $U^{(j)}\ll_{a, b, j} 1$. Let $r\in \mathbb{R}$ and $s= \sigma + i \beta \in \mathbb{C}$. We consider the following integral transform
\begin{equation*}
U^\dagger(r, s) := \int_0^{\infty} U(x) e(-rx) x^{s-1} dx. 
\end{equation*} 
We are interested in the behaviour of this integral in terms of the parameters $\beta$ and $r$. The integral  $U^\dagger (r, s) $ is of the form given in equation \eqref{eintegral} with functions
\[
g(x) = U(x) x^{\sigma-1} \ \ \ \  \textrm{and} \ \ \ \ f(x) =\frac{1}{2 \pi} \beta \log x - rx. 
\] 
We shall use the following lemma. 
 

\begin{Lemma}\cite[Lemma 5]{Mun4} \label{Fourier Mellin} 
Let $U$ be a smooth real valued function with $\textrm{supp} (U) \subset[a, b] \subset (0, \infty)$ that satisfies $U^{(j)}(x)\ll_{a, b, j} 1$. Let  $r\in \mathbb{R}$ and $s= \sigma + i \beta \in \mathbb{C}$. We have

\begin{equation*}
\begin{split}
U^{\dagger}(r,s)=\frac{\sqrt{2\pi}e(1/8)}{\sqrt{-\beta}}\left(\frac{\beta}{2\pi er}\right)^{i\beta} U_0\bigg(\sigma, \frac{\beta}{2\pi r}\bigg) + O_{a,b,\sigma}\left(\min\{|\beta|^{-3/2},|r|^{-3/2}\}\right),
\end{split}
\end{equation*}
where $U_0(\sigma,x) := x^\sigma U(x)$.
Moreover, we have the bound

\begin{equation}\label{dagger repeated ibp}
U^\dagger (r, s) = O_{a, b, \sigma, j} \left(\min \left\lbrace \left(\frac{1+|\beta|}{|r|} \right)^j , \left(\frac{1+|r|}{|\beta|} \right)^j \right\rbrace \right). 
\end{equation}
\end{Lemma}

\section{Application of circle method and dual summation formulas}
In this paper, we present calculation for Hecke-Maass cusp forms, parallel to Munshi \cite{Mun4}. Let $\pi$ be a Hecke-Maass cusp form for $\rm GL(3)$. We detect $r=n$ by using the circle method in Lemma \ref{circle method} to write $S(N)=S^+(N)+S^-(N)$, where

\begin{equation*}
\begin{split}
S^\pm(N) = &\int_0^1\underset{1\leq q\leq Q<a\leq q+Q}{\sum\sumx}\frac{1}{aq} \sum_{r\geq1}r^{-it}e\bigg(\frac{\pm r\overline{a}}{q}\bigg)e\bigg(\frac{\mp rx}{aq}\bigg)U\bigg(\frac{r}{N}\bigg) \sum_{n\geq1}\lambda(1,n) e\bigg(\frac{\mp n\overline{a}}{q}\bigg)e\bigg(\frac{\pm nx}{aq}\bigg)V\bigg(\frac{n}{N}\bigg)dx.
\end{split}
\end{equation*}

We analyze $S^+(N)$ and observe that the same bounds follow for $S^-(N)$. We start with an application of dual summation formulas to the $n$-sum and the $r$-sum.

\subsection{\texorpdfstring{Application of Poisson summation to the $r$-sum}{Application of Poisson summation to the r-sum}} The $r$-sum in above is
\begin{equation*}
\sum_{r\geq1}r^{-it}e\bigg(\frac{r\overline{a}}{q}\bigg) e\bigg(\frac{-rx}{aq}\bigg) U\bigg(\frac{r}{N}\bigg).
\end{equation*}
Breaking the $r$-sum modulo $q$ by changing variables $r\mapsto \beta+rq$, the above equals
\begin{equation*}
\sum_{r\in\BZ}\sum_{\beta\bmod q}(\beta+rq)^{-it}e\bigg(\frac{\beta\overline{a}}{q}\bigg) e\bigg(\frac{-(\beta+rq)x}{aq}\bigg) U\bigg(\frac{\beta+rq}{N}\bigg).
\end{equation*}
Applying Poisson summation to the $r$-sum and changing variables, the above sum transforms into
\begin{equation*}
N^{1-it}\sum_{r\in\BZ}\delta(\overline{a}\equiv -r\bmod q) \int_\BR u^{-it}e\bigg(\frac{-Nu(ra-x)}{aq}\bigg) U(u) du = N^{1-it}\sum_{\substack{r\in\BZ\\ \overline{a}\equiv -r\bmod q }} U^\dagger\bigg(\frac{N(ra-x)}{aq}, 1-it\bigg).
\end{equation*}
Repeated integration by parts to the $u$-integral gives arbitrary saving unless $|r|\ll qt^{1+\epsilon}/N$. The congruence condition determines $a\bmod q$. We further note than since $(a,q)=1$, the congruence forces $(r,q)=1$. Therefore $r=0$ occurs only for $q=1$, the contribution of which is negligible. For non-zero $r$ to appear, we need $q>N/t^{1-\epsilon}$. From now on, we assume $Q>N/t^{1-\epsilon}$.

\subsection{\texorpdfstring{Application of Voronoi formula to the $n$-sum}{Application of Voronoi formula to the n-sum}} The $n$-sum in $S^+(N)$ is
\begin{equation*}
\sum_{n\geq1}\lambda(1,n)e\bigg(\frac{nr}{q}\bigg)e\bigg(\frac{nx}{aq}\bigg)V\bigg(\frac{n}{N}\bigg).
\end{equation*}
Application of Voronoi summation formula as given in Lemma \ref{gl3voronoi} transforms the above sum into

\begin{equation*}
q\sum_\pm \sum_{n_1|q}\sum_{n_2=1}^\infty \frac{\lambda(n_2,n_1)}{n_1n_2} S(\overline{r},\pm n_2;q/n_1)\frI_\pm\bigg(\frac{n_1^2n_2}{q^3}, \frac{x}{aq}\bigg),
\end{equation*}
where

\begin{equation*}
\frI_\pm\bigg(\frac{n_1^2n_2}{q^3}, \frac{x}{aq}\bigg) = \frac{1}{2\pi i} \int_{(\sigma)} \bigg(\frac{n_1^2n_2N}{q^3}\bigg)^{-s}\gamma_\pm(s) V^\dagger(Nx/aq, -s)\, ds.
\end{equation*}
Let $s = \sigma + i\tau$. Using Stirling's approximation
$$ \gamma_\pm(s)\ll 1+|\tau|^{3\sigma+3/2} $$
for $\sigma\geq -1/2$. Moreover, the bound \eqref{dagger repeated ibp} of Lemma \ref{Fourier Mellin} implies
\begin{equation}\label{Vdagger}
V^\dagger(Nx/aq,-s)\ll_{\sigma, j} \min\bigg\lbrace 1, \bigg(\frac{1+|Nx/aq|}{|\tau|}\bigg)^j, \bigg(\frac{1+|\tau|}{|Nx/aq|}\bigg)^j \bigg\rbrace.
\end{equation}
Shifting the line of integration to $\sigma=M$ for large $M$ and taking $j=3M+3$, we get arbitrary saving for $n_1^2n_2\gg q^3(N/aq)^3t^\epsilon/N \sim N^2t^\epsilon/Q^3$. For the smaller values of $n_1^2n_2$, we move the contour to $\sigma=-1/2$ to write
\begin{equation*}
\begin{split}
\frI_\pm\bigg(\frac{n_1^2n_2}{q^3}, \frac{x}{aq}\bigg) = \frac{1}{2\pi} \bigg(\frac{n_1^2n_2N}{q^3}\bigg)^{1/2} \int_\BR \bigg(\frac{n_1^2n_2N}{q^3}\bigg)^{-i\tau} \gamma_\pm(-1/2+i\tau) V^\dagger(Nx/aq, 1/2-i\tau) d\tau.
\end{split}
\end{equation*}

Due to the bounds on $V^\dagger(Nx/aq, 1/2-i\tau)$, we get arbitrary saving for $|\tau|\gg Nt^\epsilon/aq$. Moreover, for $0\leq x<aqt^{\epsilon}/N$, we get arbitrary saving for $|\tau|\gg t^{2\epsilon}$ and for $aqt^{\epsilon}/N\leq x\leq1$, we get arbitrary saving for $|\tau|<1$. These observations will be used later to effectively bound certain error terms. We smoothen the $\tau$-integral by introducing a partition of unity like Munshi \cite{Mun4}. Let $\CJ$ be a collection of $O(\log t)$ many real numbers in the interval $[-Nt^\epsilon/aq, Nt^\epsilon/aq]$, containing $0$. For each $J\in\CJ$, we have a smooth function $W_J(x)$ satisfying $x^kW_J^{(k)}(x)\ll_k 1$ for $k\geq0$. Moreover, $W_0(x)$ is supported in $[-1,1]$ and satisfies the stronger bound $W_0^{(k)}(x)\ll_k1$. For each $J>0$ (resp. $J<0$), $W_J$ is supported in $[J,4J/3]$ (resp. $[4J/3,J]$). Finally, we require that
\begin{equation*}
\sum_{J\in\CJ} W_J(x) = 1 \, \text{ for } x\in [-Nt^\epsilon/aq, Nt^\epsilon/aq].
\end{equation*}
The precise definition of $W_J$ is not needed. We break the $q$-sum into dyadic segments $C\leq q<2C$ with $N/t^{1-\epsilon}\leq C\ll Q$ to write
\begin{equation*}
S^+(N) = N\sum_{\substack{N/t^{1-\epsilon}\leq C\ll Q\\ dyadic}} S(N,C) + O(t^{-2019})
\end{equation*}
where $S(N,C)$ is the following expression obtained after the above applications of dual summation formulas
\begin{equation*}
\begin{split}
S(N,C) = \frac{N^{1/2-it}}{2\pi}\sum_\pm\sum_{J\in\CJ}\sum_{n_1^2n_2\ll N^2t^\epsilon/Q^3} \frac{\lambda(n_2, n_1)}{n_2^{1/2}}\underset{\substack{C<q\leq 2C, (r,q)=1\\ 1\leq |r|\ll qt^{1+\epsilon}/N\\ n_1|q}}{\sum\sum} \frac{S(\overline{r},\pm n_2; q/n_1)}{aq^{3/2}}\frI_\pm(q,r,n_1^2n_2),
\end{split}
\end{equation*}

\begin{equation*}
\frI_\pm(q,r,n_1^2n_2) = \int_{|\tau|\ll Nt^{\epsilon}/QC} \bigg(\frac{n_1^2n_2N}{q^3}\bigg)^{-i\tau} \gamma_\pm(-1/2+i\tau) W_J(\tau) \frI^{\star\star}(q,r,\tau) d\tau,
\end{equation*}
and

\begin{equation*}
\frI^{\star\star}(q,r,\tau) = \int_0^1 V^\dagger(Nx/aq, 1/2-i\tau) U^\dagger\bigg(\frac{N(ra-x)}{aq}, 1-it\bigg) dx.
\end{equation*}

Next, we analyze the above integrals using stationary phase analysis.

\section{\texorpdfstring{Analysis of the integrals}{Analysis of the integrals}}

Using Lemma \ref{Fourier Mellin}, we get the asymptotic estimate
\begin{equation*}
U^\dagger\bigg(\frac{N(ra-x)}{aq}, 1-it\bigg) = \frac{\sqrt{2\pi}e(1/8)}{\sqrt{t}}\bigg(\frac{-taq}{2\pi eN(ra-x)}\bigg)^{-it} U_0\bigg(1, \frac{-taq}{2\pi N(ra-x)}\bigg) + O(t^{-3/2+\epsilon}).
\end{equation*}
We recall that we get arbitrary saving for $0<x<aqt^{\epsilon}/N$ and $|\tau|\gg t^{2\epsilon}$. Using the above asymptotic and the bound \eqref{Vdagger} for $V^\dagger(Nx/aq, 1/2-i\tau)$ in this range, we get
\begin{equation*}
\int_0^{aqt^{\epsilon}/N} V^\dagger(Nx/aq, 1/2-i\tau) U^\dagger\bigg(\frac{N(ra-x)}{aq}, 1-it\bigg) dx \ll_j \frac{QCt^{\epsilon}}{Nt^{1/2}}\bigg(\frac{t^{\epsilon}}{1+|\tau|}\bigg)^{j}.
\end{equation*}
For $aqt^{\epsilon}/N<x\leq1$ (so that $|\tau|>t^{\epsilon/2}$, otherwise we get arbitrary saving), we use Lemma \ref{Fourier Mellin} to write

\begin{equation*}
V^\dagger\bigg(\frac{Nx}{aq}, 1/2-i\tau\bigg) = \frac{\sqrt{2\pi}e(1/8)}{\sqrt{\tau}}\bigg(\frac{-\tau aq}{2\pi eNx}\bigg)^{-i\tau} V_0\bigg(\frac{1}{2}, \frac{-\tau aq}{2\pi Nx}\bigg) + O\bigg(\min\bigg\lbrace\bigg|\frac{Nx}{aq}\bigg|^{-3/2}, |\tau|^{-3/2}\bigg\rbrace\bigg).
\end{equation*}
Therefore for an absolute constant $c_1$,
\begin{equation}\label{after first sp}
\begin{split}
\frI^{\star\star}(q,r,\tau) = \frac{c_1}{\sqrt{t\tau}}\bigg(\frac{-taq}{2\pi eN}\bigg)^{-it}\bigg(\frac{-\tau aq}{2\pi eN}\bigg)^{-i\tau} &\int_{aqt^{\epsilon}/N}^1 (ra-x)^{it} x^{i\tau} U_0\bigg(1, \frac{-taq}{2\pi N(ra-x)}\bigg) V_0\bigg(\frac{1}{2}, \frac{-\tau aq}{2\pi Nx}\bigg) dx \\
&+ O_j\bigg(E^{\star\star} + t^{-3/2+\epsilon} +  \frac{QCt^{\epsilon}}{Nt^{1/2}}\bigg(\frac{t^{\epsilon}}{1+|\tau|}\bigg)^{j}\bigg),
\end{split}
\end{equation}
where 

\begin{equation*}
E^{\star\star} = \frac{1}{t^{1/2}}\int_{aqt^{\epsilon}/N}^1 \min\bigg\lbrace\bigg|\frac{Nx}{aq}\bigg|^{-3/2}, |\tau|^{-3/2}\bigg\rbrace dx,
\end{equation*}
and the main term and the error term $E^{\star\star}$ occur only for $|\tau|>1$. We observe that $E^{\star\star}$ can be bounded as
\begin{equation*}
\begin{split}
E^{\star\star} &= \frac{1}{t^{1/2}}\int_{aqt^{\epsilon}/N}^{\min\{1, |\tau|aq/N\}} |\tau|^{-3/2}dx + \frac{1}{t^{1/2}}\int_{|\tau|aq/N}^1 \bigg(\frac{Nx}{aq}\bigg)^{-3/2} \delta(|\tau|aq/N < 1) dx\\
&\ll \frac{1}{t^{1/2}|\tau|^{3/2}}\min\bigg\lbrace 1, \frac{|\tau|aq}{N}\bigg\rbrace + \frac{1}{t^{1/2}|\tau|^{1/2}}\frac{aq}{N}\delta(|\tau|aq/N<1) \ll \frac{1}{t^{1/2}|\tau|^{3/2}}\min\bigg\lbrace 1, \frac{|\tau|aq}{N}\bigg\rbrace.
\end{split}
\end{equation*}

\subsection{\texorpdfstring{Analysis of the $x$-integral}{Analysis of the x-integral}} To analyze the integral in \eqref{after first sp}, we set $\frI = \int_{aqt^\epsilon/N}^1e(f(x))g(x)dx$ with
\begin{equation*}
\begin{split}
&f(x) = \frac{t\log(ra-x)}{2\pi} + \frac{\tau\log x}{2\pi},\\
& g(x) = U_0\bigg(1, \frac{-taq}{2\pi N(ra-x)}\bigg) V_0\bigg(\frac{1}{2}, \frac{-\tau aq}{2\pi Nx}\bigg).
\end{split}
\end{equation*}
Due to the bounds $x^{j}U^{(j)}(x)\ll_j 1$ and $x^{j}V^{(j)}(x)\ll_j 1$, and the definitions of $U_0$ and $V_0$ as given in Lemma \ref{Fourier Mellin}, we have
\begin{equation*}
g^{(j)}(x)\ll_j 1.
\end{equation*}
Since $supp\, V^\dagger(1/2,x)\subset[1,2]$, the support of $g(x)$ lies inside $[-\tau aq/4\pi N, -\tau aq/2\pi N]$. This lies inside $[aqt^\epsilon/N,1]$ for $-\tau\in[4\pi t^\epsilon, 2\pi N/aq]$, while $g(x)=0$ for $x\in[aqt^\epsilon/N,1]$ and $-\tau\geq 4\pi N/aq$. For $-\tau<4\pi t^\epsilon$, we use the second derivative test \eqref{rth der test}. For $-\tau\in[2\pi N/aq, 4\pi N/aq]$, we need to be a little more careful in the analysis of integrals since the $r^{th}$ derivative test \eqref{rth der test} is not sufficient. 
For $j\geq1$,
\begin{equation*}
2\pi f^{(j)}(x) = \frac{-t(j-1)!}{(ra-x)^j} + \frac{\tau(-1)^{j-1}(j-1)!}{x^j}.
\end{equation*}
In the support of the integral, $x-ra \sim taq/2\pi N$ and $x\sim -\tau aq/2\pi N$, where $a\sim b$ means $k_1<a/b<k_2$ for some constants $k_1,k_2>0$. Since $Q>N/t^{1-\epsilon}$, we have
\begin{equation}\label{jth der}
f^{(j)}(x) \sim_j -\tau\bigg(\frac{N}{\tau aq}\bigg)^j \quad \text{ for } j\geq2
\end{equation}
where $a\sim_j b$ means that the constants $k_1, k_2$ depend on $j$. Therefore $|f''(x)|\sim -\tau(N/\tau QC)^2$. The point $x_0$ where $f'(x_0)=0$ is called a stationary point. We have
\begin{equation*}
x_0 = \frac{ra\tau}{\tau+t}.
\end{equation*}
We recall that due to \eqref{Vdagger}, we have $\frI = O(t^{-20200})$ for $x>aqt^{\epsilon}/N$ and $|\tau|< t^{\epsilon/2}$. For $-\tau\in[t^{\epsilon/2},4\pi t^{\epsilon}]$, we apply the second derivative bound \eqref{rth der test} and the estimate \eqref{jth der} for $j=2$ along with the observation that $Var_{x\in[aqt^\epsilon/N,1]}\, g(x)\ll t^\epsilon$ to bound
\begin{equation}\label{easy bound 1}
\frI \ll \frac{QCt^\epsilon}{N}.
\end{equation}
For larger $|\tau|$, we start by writing the Taylor expansion of $f'(x)$ around $x=x_0$,
\begin{equation}\label{Taylor}
f'(x) = (x-x_0)f''(x_0) + O\bigg((x-x_0)^2\tau\bigg(\frac{N}{\tau QC}\bigg)^3\bigg).
\end{equation}
The error term follows from the estimate \eqref{jth der} for $j\geq3$. 

We now analyze the case $-\tau\in[4\pi t^\epsilon, 2\pi N/aq]$, for which $g(aqt^\epsilon/N)=g(1)=0$. We can therefore change the limits of the integral $\frI$ to write
\begin{equation*}
\frI = \int_{-\tau aq/4\pi N}^{-\tau aq/2\pi N} g(x)e(f(x))dx.
\end{equation*}
In case $x_0$ lies inside the interval $[-\tau aq/4\pi N, -\tau aq/2\pi N]$, one can expand the interval of integration to $[-\tau aq/8\pi N, -\tau aq/\pi N]$ without changing $\frI$. Then $\kappa = \min\{x_0 +\tau aq/8\pi N, -\tau aq/\pi N-x_0 \}\gg |\tau|QC/N$. Applying the second statement of Lemma \ref{huxley} with 
\begin{equation*}
\Theta_f = |\tau|, \quad \Omega_f = |\tau|QC/N, \quad \kappa = |\tau|QC/N, \quad \Omega_g=1,
\end{equation*}
so that the hypothesis $\Omega_f\gg (b-a)$ of Lemma \ref{huxley} is satisfied, we obtain
\begin{equation}\label{easy bound 2}
\frI = \frac{g(x_0)e(f(x_0)+1/8)}{\sqrt{f''(x_0)}} + O\bigg(\frac{QC}{|\tau|^{1/2}N}\bigg).
\end{equation}
In case $x_0$ does not lie inside the interval $I=[-\tau aq/4\pi N, -\tau aq/2\pi N]$, \eqref{Taylor} implies 
\begin{equation*}
\Lambda = \min_{x\in I}|f'(x)| \sim \min|x-x_0|N^2/|\tau|Q^2C^2.
\end{equation*}
In the case $\min_{x\in I}|x-x_0|>t^\epsilon\sqrt{|\tau|}QC/N$, we use the above estimate for $\min_{x\in I}|f'(x)|$ and apply Lemma \ref{bky no sp} to obtain $\frI = O(t^{-20200})$. While in the case $\min_{x\in I}|x-x_0|<t^\epsilon\sqrt{|\tau|}QC/N$, we expand the interval $I$ to $[-\tau aq/8\pi N, -\tau aq/\pi N]$ without changing $\frI$, so that $x_0$ now lies in the expanded interval with $\kappa\gg |\tau|QC/N$. The analysis now is the same as in the previous case. Putting together the estimates \eqref{easy bound 1} and \eqref{easy bound 2}, we obtain that for $1\leq|\tau|\leq 2\pi N/aq$,
\begin{equation}\label{easy asymp}
\frI = \frac{g(x_0)e(f(x_0)+1/8)}{\sqrt{f''(x_0)}} + O\bigg(\frac{QC}{|\tau|^{1/2}N}\bigg).
\end{equation}

Next we analyze the case $-\tau\in[2\pi N/aq, 4\pi N/aq]$. We start by observing that in this case, $g(aqt^{\epsilon}/N)=0$ but $g(1)\neq0$. We will therefore have to divide our analysis depending on the size of $f'(1)$. In view of \eqref{Taylor}, this translates to the size of $|x-x_0|$. Let $\kappa:= x_0-1$ (so that $x_0$ is outside $[aqt^\epsilon/N,1]$ if $\kappa>0$). Since $x_0 = ra\tau/(t+\tau)$ and $ra\sim -taq/N$ (due to the support of $U_0(1,x)$ being in $x\in[1/2, 5/2]$), we have $x_0 \sim -aq\tau/N$. Therefore $|x_0 - aqt^\epsilon/N|\gg 1$.

We observe that $x_0 = ra\tau/(t+\tau)$ implies $\tau = x_0t/(ra-x_0)$. In particular, $x_0=1$ for $\tau_0 := t/(ra-1)$. Since $x_0=1+\kappa$ by definition, \begin{equation}\label{tau in kappa}
\tau = \frac{(1+\kappa)t}{ra-(1+\kappa)} = \frac{(1+\kappa)t}{ra-1} + O\bigg(\frac{(1+\kappa)t\kappa}{(ra-1)^2}\bigg) = \tau_0 + \tau_0\kappa\bigg(1 + O\bigg(\frac{N}{tQC}\bigg)\bigg).
\end{equation}
Let $\kappa_0:= t^\epsilon\sqrt{QC/N}$. If $|\kappa|<\kappa_0$, we apply the second derivative bound \eqref{rth der test} and the estimate \eqref{jth der} for $j=2$ along with the observation that $Var_{x\in[aqt^\epsilon/N,1]}\, g(x)\ll t^\epsilon$ to obtain 
\begin{equation*}
\frI \ll \frac{QC\sqrt{|\tau|}t^\epsilon}{N}. 
\end{equation*}
Since $|\tau|\sim N/QC$, $g(x)\ll 1$ and $f''(x)\sim N^2/Q^2C^2|\tau|$ in the support of $g(x)$, the above bound can be replaced by \begin{equation}\label{critical bound 1}
\frI \ll \frac{g(x_0)e(f(x_0)+1/8)}{\sqrt{f''(x_0)}} +  t^\epsilon\sqrt{\frac{QC}{N}}. 
\end{equation}
Next, if $\kappa>\kappa_0$ (so that $x_0>1+\kappa_0$ and therefore does not lie in $I=[aqt^\epsilon/N, 1]$), we have $\min_{x\in I}|f'(x)| = N\kappa/QC$ in view of \eqref{Taylor}. Applying the first statement of Lemma \ref{huxley} with
\begin{equation*}
\Lambda = \frac{N\kappa}{QC}, \qquad \Theta_f = \frac{N}{QC}, \qquad \Omega_f=\Omega_g = 1,
\end{equation*}
along with the observations $\kappa> \kappa_0=t^\epsilon\sqrt{QC/N}$ and $|f'(1)|\gg N\kappa/QC$, we obtain
\begin{equation}\label{critical bound 2}
\frI \ll \frac{QC}{N\kappa}.
\end{equation}
Finally, when $\kappa<-\kappa_0$ (so that $x_0$ lies inside $I=[aqt^\epsilon/N, 1]$), we apply the second statement of Lemma \ref{huxley} with
$\Theta_f = N/QC,\, \Omega_f=\Omega_g=1$, and use $|\kappa|>t^\epsilon\sqrt{QC/N}$ and $|f'(1)|\gg N\kappa/QC$ to obtain
\begin{equation}\label{critical bound 3}
\frI = \frac{g(x_0)e(f(x_0)+1/8)}{\sqrt{f''(x_0)}} + O\bigg(\frac{QC}{N|\kappa|}\bigg).
\end{equation}
Putting together \eqref{critical bound 1}, \eqref{critical bound 2} and \eqref{critical bound 3}, we obtain that for $-\tau\in[2\pi N/aq, 4\pi N/aq]$,
\begin{equation}\label{critical asymp}
\frI = \frac{g(x_0)e(f(x_0)+1/8)}{\sqrt{f''(x_0)}} + O\bigg(\min\bigg\{\frac{QC}{N|\kappa|}, \kappa_0\bigg\}\bigg).
\end{equation}
In view of the change of variable \eqref{tau in kappa}, the estimate $N/tQC = o(1)$ and $\tau_0\sim N/QC$, we can replace the above error term by $O(\min\{|\tau-\tau_0|\-, \kappa_0\})$.

Putting together the estimates \eqref{easy asymp} and \eqref{critical asymp}, we obtain that for $1\leq|\tau|\ll Nt^\epsilon/QC$,
\begin{equation*}
\frI = \frac{g(x_0)e(f(x_0)+1/8)}{\sqrt{f''(x_0)}} + O\bigg(\frac{QC}{|\tau|^{1/2}N} + \delta(|\tau|\sim N/QC) \min\{|\tau-\tau_0|\-, \kappa_0\} \bigg).
\end{equation*}

Explicit computations show that
\begin{equation*}
\begin{split}
2\pi f(x_0) &= t\log(t/\tau) + (t+\tau)\log\bigg(\frac{ra\tau}{t+\tau}\bigg)\\
2\pi f''(x_0) &= \frac{(t+\tau)^3}{t\tau r^2a^2}, \text{ and }\\
g(x_0) &= U_0\bigg(1, \frac{-q(t+\tau)}{2\pi Nr}\bigg)V_0\bigg(\frac{1}{2}, \frac{-q(t+\tau)}{2\pi Nr}\bigg)=V_0\bigg(\frac{3}{2}, \frac{-q(t+\tau)}{2\pi Nr}\bigg).
\end{split}
\end{equation*}
We recall that $U(x)V(x)=V(x)$, therefore $U_0(r_1,x)V_0(r_2,x)=V_0(r_1+r_2,x)$. The above calculations therefore show that the expression in \eqref{after first sp} is asymptotic to
\begin{equation*}
c_2\frac{ra}{(t+\tau)^{3/2}}\bigg(\frac{-(t+\tau)q}{2\pi eNr}\bigg)^{-i(t+\tau)}V_0\bigg(\frac{3}{2}, \frac{-q(t+\tau)}{2\pi Nr}\bigg),
\end{equation*}
for come constant $c_2$. We collect these calculations in the following lemma.

\begin{Lemma}\label{sp lemma}
Let $N/t^{1-\epsilon}<Q$. Then
\begin{equation*}
\frI^{\star\star}(q,r,\tau) = \frI_1(q,r,\tau) + \frI_2(q,r,\tau)
\end{equation*}
with,
\begin{equation*}
\frI_1(q,r,\tau) = c_2\frac{ra}{(t+\tau)^{3/2}}\bigg(\frac{-(t+\tau)q}{2\pi eNr}\bigg)^{-i(t+\tau)}V_0\bigg(\frac{3}{2}, \frac{-q(t+\tau)}{2\pi Nr}\bigg),
\end{equation*}
and $\frI_2(q,r,\tau) = O(B(C,\tau))$ where
\begin{align*}
B(C,\tau) = &\frac{QC}{Nt^{1/2}}\bigg(\frac{t^\epsilon}{1+|\tau|}\bigg)^{10} + t^{-3/2+\epsilon} + \bigg(E^{\star\star}+ \frac{QC}{t^{1/2}|\tau|N}\bigg)\delta(|\tau|>1) \\ &+ \delta(|\tau|\sim N/QC)\frac{1}{t^{1/2}|\tau|^{1/2}} \min\{|\tau-\tau_0|\-, \kappa_0\}.
\end{align*}
\end{Lemma}

We observe that since $N/t^{1-\epsilon}<Q<N^{1/2}$,
\begin{equation}\label{intB(C,tau)}
\int_{|\tau|\ll Nt^\epsilon/QC} B(C,\tau)\ll t^\epsilon\sqrt{QC/Nt}.
\end{equation}

\begin{Remark}
We observe that $\frI_1(q,r,\tau)\ll QCt^\epsilon/Nt^{1/2}$. This is slightly better than the corresponding trivial bound of $1/t^{1/2}K$ obtained by Munshi in Lemma $8$ of \cite{Mun4}. This comparison is made by substituting $Q=(N/K)^{1/2}$. This will finally help us to improve upon Munshi's bound.

\end{Remark}

Using Lemma \ref{sp lemma}, we have the following decomposition of $S(N,C)$.

\begin{Lemma}\label{decomposition into S1 and S2 lemma}
\begin{equation*}
S(N,C) = \sum_{J\in\CJ} S_{1,J}(N,C) + S_{2,J}(N,C)
\end{equation*}
where
\begin{equation*}
S_{\ell, J}(N,C) = \frac{N^{1/2-it}}{2\pi} \sum_{\pm} \sum_{n_1^2n_2\ll N^2t^\epsilon/Q^3} \frac{\lambda(n_2,n_1)}{n_2^{1/2}} \underset{\substack{C<q\leq 2C, (r,q)=1\\ 1\leq |r|\ll qt^{1+\epsilon}/N\\ n_1|q}}{\sum\sum} \frac{S(\overline{r}, \pm n_2; q/n_1)}{aq^{3/2}}\frI_{\ell,J, \pm}(q,r,n_1^2n_2)
\end{equation*}
and
\begin{equation*}
\frI_{\ell, J, \pm}(q,r,n) = \int_\BR \bigg(\frac{nN}{q^3}\bigg)^{-i\tau} \gamma_\pm(-1/2+i\tau)\frI_\ell(q,r,\tau) W_J(\tau) d\tau,
\end{equation*}
where $\frI_\ell(q,r,\tau)$ is defined in Lemma \ref{sp lemma}.
\end{Lemma}

\section{Cauchy-Schwarz and Poisson summation- First Application}

In this section, we analyze 
\begin{equation*}
S_2(N,C):= \sum_{J\in\CJ} S_{2, J}(N,C).
\end{equation*}
Breaking the $n$-sum into dyadic segments of length $L$
\begin{equation*}
\begin{split}
S_{2}(N,C)\ll t^\epsilon N^{1/2}\int_{|\tau|<\frac{Nt^\epsilon}{QC}} \sum_\pm \sum_{\substack{1\leq L\ll N^2t^\epsilon/Q^3\\ dyadic}}& \sum_{n_1, n_2} \frac{|\lambda(n_2,n_1)|}{n_2^{1/2}} U\bigg(\frac{n_1^2n_2}{L}\bigg) \\
&\times\bigg|\underset{\substack{C<q\leq 2C, (r,q)=1\\ 1\leq |r|\ll qt^{1+\epsilon}/N\\ n_1|q}}{\sum\sum} \frac{S(\overline{r},\pm n_2; q/n_1)}{aq^{3/2-3i\tau}}\frI_2(q,r,\tau)\bigg|d\tau.
\end{split}
\end{equation*}

Next we apply Cauchy-Schwarz inequality with the $n_1, n_2$-sums outside and apply the Ramanujan bound on average (Lemma \ref{Ram bound}) to write

\begin{equation}\label{error after cauchy}
S_2(N,C)\ll t^\epsilon N^{1/2} \int_{|\tau|<\frac{Nt^\epsilon}{QC}} \sum_\pm \sum_{\substack{1\leq L\ll N^2t^\epsilon/Q^3\\ dyadic}} L^{1/2}[S_{2,\pm}(N,C,L,\tau)]^{1/2} d\tau,
\end{equation}
where
\begin{equation*}
S_{2,\pm}(N,C,L,\tau) = \sum_{n_1, n_2} \frac{1}{n_2}U\bigg(\frac{n_1^2n_2}{L}\bigg)\bigg|\underset{\substack{C<q\leq 2C, (r,q)=1\\ 1\leq |r|\ll qt^{1+\epsilon}/N\\ n_1|q}}{\sum\sum} \frac{S(\overline{r},\pm n_2; q/n_1)}{aq^{3/2-3i\tau}}\frI_2(q,r,\tau)\bigg|^2.
\end{equation*}

The analysis of $S_{2,-}(N,C,L,\tau)$ is similar to that of $S_{2,+}(N,C,L,\tau)$, so we consider only\\ $S_{2,+}(N,C,L,\tau)$. Opening the absolute value squared the expression of $S_{2,+}(N,C,L,\tau)$ is
\begin{equation}\label{S2NCLtau}
\sum_{n_1\leq2C} \underset{\substack{C<q_1\leq 2C, (r_1,q_1)=1\\ 1\leq |r_1|\ll q_1t^{1+\epsilon}/N\\ n_1|q_1}}{\sum\sum}\  \underset{\substack{C<q_2\leq 2C, (r_2,q_2)=1\\ 1\leq |r_2|\ll q_2t^{1+\epsilon}/N\\ n_1|q_2}}{\sum\sum}\  \frac{1}{a_1a_2q_1^{3/2-3i\tau}q_2^{3/2+3i\tau}}\frI_2(q_1,r_1,\tau)\overline{\frI_2(q_2,r_2,\tau)}\BT
\end{equation}
where we temporarily set $\BT$ to include the $n_2$-sum
\begin{equation*}
\BT= \sum_{n_2}\frac{1}{n_2}U\bigg(\frac{n_1^2n_2}{L}\bigg)S(\overline{r_1},n_2; q_1/n_1)S(\overline{r_2},n_2; q_2/n_1).
\end{equation*}

Let $\hat{q}_1=q_1/n_1$ and $\hat{q}_2=q_2/n_1$. Breaking the $n_2$-sum modulo $\hat{q}_1\hat{q}_2$ and applying Poisson summation formula as Munshi does, $\BT$ transforms into
\begin{equation*}
\BT= \frac{n_1^2}{q_1q_2}\sum_{n_2\in\BZ} \frC U^\dagger(n_2L/q_1q_2,0),
\end{equation*}
where $\frC$ is the sum
\begin{equation}\label{character sum}
\frC = \sum_{\beta\bmod \hat{q}_1\hat{q}_2} S(\overline{r_1},\beta; q_1/n_1)S(\overline{r_2},\beta; q_2/n_1) e(\beta n_2/\hat{q}_1\hat{q}_2).
\end{equation}
Bounds on $U^\dagger(n_2L/q_1q_2,0)$ give arbitrary saving for $|n_2|\gg C^2t^\epsilon/L$. Recalling that $a$ is of size $Q$, the expression of $S_{2,+}(N,C,L,\tau)$ in \eqref{S2NCLtau} is bounded by
\begin{equation*}
\frac{B(C,\tau)^2}{Q^2C^5}\sum_{n_1\leq2C} \underset{\substack{C<q_1\leq 2C, (r_1,q_1)=1\\ 1\leq |r_1|\ll q_1t^{1+\epsilon}/N\\ n_1|q_1}}{\sum\sum}\  \underset{\substack{C<q_2\leq 2C, (r_2,q_2)=1\\ 1\leq |r_2|\ll q_2t^{1+\epsilon}/N\\ n_1|q_2}}{\sum\sum}\ n_1^2 \sum_{|n_2|\ll C^2t^\epsilon/L} |\frC| + O(t^{-2019}).
\end{equation*}
We use Lemma $13$ of \cite{Mun4} to bound the character sum $\frC$. For completeness, we state it here.

\begin{Lemma}\label{character sum lemma}
We have
$$ \frC\ll \hat{q}_1\hat{q}_2 (\hat{q}_1,\hat{q}_2,n_2).$$
Moreover for $n_2=0$, we get that $\frC=0$ unless $\hat{q}_1=\hat{q}_2$, in which case
$$ \frC\ll \hat{q}_1^2(\hat{q}_1, r_1-r_2). $$
\end{Lemma}

\subsection{Diagonal contribution} We first consider the contribution of $n_2=0$, which we denote by\\ $S_{2,+}^\flat(N,C,L,\tau)$. Using the second statement of Lemma \ref{character sum lemma}
\begin{equation*}
S_{2,+}^\flat(N,C,L,\tau)\ll \frac{B(C,\tau)^2}{Q^2C^5} \sum_{n_1\leq 2C} \bigg\lbrace \frac{C^5t}{n_1^2N} + \frac{C^5t^2}{n_1N^2}\bigg\rbrace,
\end{equation*}
where the first term is the contribution from terms with $r_1=r_2$ and the second term is the contribution from $r_1\neq r_2$. Since we will choose $N>t^{1+\epsilon}$, the first term dominates and we get
\begin{equation}\label{error diagonal}
S_{2,+}^\flat(N,C,L,\tau)\ll \frac{t^{1+\epsilon}}{NQ^2}B(C,\tau)^2.
\end{equation}

\subsection{Off-diagonal contribution} We now bound the contribution of terms with $n_2\neq0$, which we denote by $S_{2,+}^\sharp(N,C,L,\tau)$. Using the first statement of Lemma \ref{character sum lemma},
\begin{equation}\label{error off-diagonal}
\begin{split}
S_{2,+}^\sharp(N,C,L,\tau)&\ll \frac{B(C,\tau)^2}{Q^2C^3} \sum_{n_1\leq2C} \underset{\substack{C<q_1\leq 2C, (r_1,q_1)=1\\ 1\leq |r_1|\ll q_1t^{1+\epsilon}/N\\ n_1|q_1}}{\sum\sum}\  \underset{\substack{C<q_2\leq 2C, (r_2,q_2)=1\\ 1\leq |r_2|\ll q_2t^{1+\epsilon}/N\\ n_1|q_2}}{\sum\sum}\ n_1^2 \sum_{1\leq|n_2|\ll C^2t^\epsilon/L} (q_1,n_2)\\
&\ll \frac{t^{2+\epsilon}C^3}{Q^2N^2L}B(C,\tau)^2.
\end{split}
\end{equation}

\subsection{Estimating $S_2(N,C)$} Using the bounds \eqref{error diagonal} and \eqref{error off-diagonal} in \eqref{error after cauchy},
\begin{equation*}
\begin{split}
S_2(N,C)\ll t^\epsilon N^{1/2} \int_{|\tau|<\frac{Nt^\epsilon}{QC}} B(C,\tau) \sum_\pm \sum_{\substack{1\leq L\ll N^2t^\epsilon/Q^3\\ dyadic}} \bigg\lbrace \frac{L^{1/2}t^{1/2}}{N^{1/2}Q} + \frac{tC^{3/2}}{NQ}\bigg\rbrace d\tau.
\end{split}    
\end{equation*}
Using the estimate \eqref{intB(C,tau)} in above
\begin{equation*}
S_2(N,C)\ll t^\epsilon{Q^{1/2}C^{1/2}}\bigg\{ \frac{N^{1/2}}{Q^{5/2}} + \frac{t^{1/2}C^{3/2}}{NQ}\bigg\}.
\end{equation*}
Multiplying by N and summing over $C\ll Q$ dyadically, the contribution of $S_2(N,C)$ towards $S^+(N)$ is
\begin{equation}\label{S2}
t^\epsilon NQ\bigg\{ \frac{N^{1/2}}{Q^{5/2}} + \frac{t^{1/2}Q^{1/2}}{N}\bigg\}.
\end{equation}

\section{Cauchy-Schwarz and Poisson Summation- Second application}
We now analyze the sum $S_{1,J}(N,C)$. We need to get further cancellations in the $\tau$-integral to get a subconvexity bound. For notational simplicity, we consider only the positive $J$ with $J\gg t^\epsilon$. The same analysis holds for negative values of $J$ with $-J\gg t^\epsilon$. For $|J|\ll t^\epsilon$, the analysis is done as before since we do not need to get cancellation in the $\tau$-integral. We recall that

\begin{equation*}
S_{1,J}(N,C) = \frac{N^{1/2-it}}{2\pi} \sum_{\pm} \sum_{n_1^2n_2\ll N^2t^\epsilon/Q^3} \frac{\lambda(n_2,n_1)}{n_2^{1/2}} \underset{\substack{C<q\leq 2C, (r,q)=1\\ 1\leq |r|\ll qt^{1+\epsilon}/N\\ n_1|q}}{\sum\sum} \frac{S(\overline{r},\pm n_2; q/n_1)}{aq^{3/2}}\frI_{1,\pm,J}(q,r,n_1^2n_2)
\end{equation*}
where
\begin{equation*}
\frI_{1,\pm,J}(q,r,n) = \int_\BR \bigg(\frac{n_1^2n_2N}{q^3}\bigg)^{-i\tau} \gamma_\pm(-1/2+i\tau)\frI_1(q,r,\tau) W_J(\tau) d\tau.
\end{equation*}
Taking absolute values while keeping the $\tau$-integral inside,
\begin{equation*}
\begin{split}
S_{1,J}(N,C)\ll &t^\epsilon N^{1/2} \sum_\pm \sum_{\substack{1\leq L\ll N^2t^\epsilon/Q^3\\ dyadic}} \sum_{n_1, n_2} \frac{|\lambda(n_2,n_1)|}{n_2^{1/2}} U\bigg(\frac{n_1^2n_2}{L}\bigg) \\
&\times\bigg|\int_{\BR}(n_1^2n_2N)^{-i\tau}\gamma_{\pm}(-1/2+i\tau)\underset{\substack{C<q\leq 2C, (r,q)=1\\ 1\leq |r|\ll qt^{1+\epsilon}/N\\ n_1|q}}{\sum\sum} \frac{S(\overline{r},\pm n_2; q/n_1)}{aq^{3/2-3i\tau}}\frI_1(q,r,\tau)W_J(\tau) d\tau\bigg|.
\end{split}
\end{equation*}
Applying Cauchy-Schwarz inequality by pulling the $n_1, n_2$-sums outside along with the Ramanujan bound on average (Lemma \ref{Ram bound}),
\begin{equation}\label{main after cauchy}
S_{1,J}(N,C)\ll t^\epsilon N^{1/2} \sum_\pm \sum_{\substack{1\leq L\ll N^2t^\epsilon/Q^3\\ dyadic}} L^{1/2}[S_{1,\pm,J}(N,C,L)]^{1/2}
\end{equation}
where

\begin{equation*}
\begin{split}
S_{1,\pm,J}(N,C,L) = \sum_{n_1, n_2} \frac{1}{n_2}U\bigg(\frac{n_1^2n_2}{L}\bigg)\bigg|\int_{\BR}& (n_1^2n_2N)^{-i\tau}\gamma_{\pm}(-1/2+i\tau)\\
&\times \underset{\substack{C<q\leq 2C, (r,q)=1\\ 1\leq |r|\ll qt^{1+\epsilon}/N\\ n_1|q}}{\sum\sum} \frac{S(\overline{r},n_2; q/n_1)}{aq^{3/2-3i\tau}}\frI_1(q,r,\tau)W_J(\tau)d\tau\bigg|^2.
\end{split}
\end{equation*}
Like earlier, we consider only $S_{1,+,J}(N,C,L)$. Opening absolute value squared,

\begin{equation*}
\begin{split}
&S_{1,+,J}(N,C,L) = \sum_{n_1\leq2C}\underset{\BR^2}{\int\int} (n_1^2N)^{i(\tau_2-\tau_1)}\gamma_+(-1/2+i\tau_1)\overline{\gamma_+(-1/2+i\tau_2)}W_J(\tau_1)W_J(\tau_2) \\
&\times \underset{\substack{C<q_1\leq 2C, (r_1,q_1)=1\\ 1\leq |r_1|\ll Ct^{1+\epsilon}/N\\ n_1|q_1}}{\sum\sum}\ \underset{\substack{C<q_2\leq 2C, (r_2,q_2)=1\\ 1\leq |r_2|\ll Ct^{1+\epsilon}/N\\ n_1|q_2}}{\sum\sum} \frac{1}{a_1a_2q_1^{3/2-3i\tau_1}q_2^{3/2+3i\tau_2}} \frI_1(q_1,r_1,\tau_1)\overline{\frI_1(q_2,r_2,\tau_2)}\ \BT\ d\tau_1\ d\tau_2, 
\end{split}
\end{equation*}
where we temporarily set
\begin{equation*}
\BT= \sum_{n_2\in\BZ}\frac{1}{n_2^{1+i(\tau_1-\tau_2)}}U\bigg(\frac{n_1^2n_2}{L}\bigg)S(\overline{r_1},n_2; q_1/n_1)S(\overline{r_2},n_2; q_2/n_1).
\end{equation*}
Like before, breaking the $n_2$-sum $\bmod\ q_1q_2/n_1^2$ and applying Poisson summation to it
\begin{equation*}
\BT = \frac{n_1^2}{q_1q_2}\bigg(\frac{L}{n_1^2}\bigg)^{i(\tau_2-\tau_1)}\sum_{n_2\in\BZ}\frC\ U^\dagger(n_2L/q_1q_2, i(\tau_2-\tau_1)).
\end{equation*}
Here $\frC$ is the character sum as given in \eqref{character sum}. Since $|\tau_i|\ll Nt^\epsilon/CQ$, $U^\dagger$ gives arbitrary saving for $|n_2|\gg NCt^\epsilon/QL$. Recalling that $a\sim Q$, $S_{1,+,J}(N,C,L)$ is bounded by
\begin{equation}\label{S1jNCL}
\frac{t^\epsilon}{Q^2C^5} \sum_{n_1\leq2C} \underset{\substack{C<q_1\leq 2C, (r_1,q_1)=1\\ 1\leq |r_1|\ll q_1t^{1+\epsilon}/N\\ n_1|q_1}}{\sum\sum}\  \underset{\substack{C<q_2\leq 2C, (r_2,q_2)=1\\ 1\leq |r_2|\ll q_2t^{1+\epsilon}/N\\ n_1|q_2}}{\sum\sum}\ n_1^2 \sum_{|n_2|\ll NCt^\epsilon/QL} |\frC|\ |\frK| + O(t^{-2019}),
\end{equation}
where
\begin{equation*}
\begin{split}
\frK = \underset{\BR^2}{\int\int}& \gamma_+(-1/2+i\tau_1)\overline{\gamma_+(-1/2+i\tau_2)}W_J(\tau_1)W_J(\tau_2)\frac{(LN)^{i(\tau_2-\tau_1)}}{q_1^{-3i\tau_1}q_2^{3i\tau_2}}\\ &\times \frI_1(q_1,r_1,\tau_1)\overline{\frI_1(q_2,r_2,\tau_2)} U^\dagger(n_2L/q_1q_2, i(\tau_2-\tau_1))\ d\tau_1\ d\tau_2.
\end{split}
\end{equation*}

\subsection{Analysis of the integral $\frK$} Using the expression for $\frI_1(q,r,\tau)$ as given in Lemma \ref{sp lemma}
\begin{equation*}
\begin{split}
\frK = &|c_2|^2 \frac{r_1r_2a_1a_2}{t^2}\underset{\BR^2}{\int\int} \gamma_+(-1/2+i\tau_1)\overline{\gamma_+(-1/2+i\tau_2)}W_J(q_1,r_1,\tau_1)W_J(q_2,r_2,\tau_2)\frac{(LN)^{i(\tau_2-\tau_1)}}{q_1^{-3i\tau_1}q_2^{3i\tau_2}}\\
&\times \bigg(\frac{-(t+\tau_1)q_1}{2\pi eNr_1}\bigg)^{-i(t+\tau_1)}\bigg(\frac{-(t+\tau_2)q_2}{2\pi eNr_2}\bigg)^{i(t+\tau_2)}U^\dagger(n_2L/q_1q_2, i(\tau_2-\tau_1))\ d\tau_1\ d\tau_2,
\end{split}
\end{equation*}
where
\begin{equation*}
W_J(q,r,\tau) = \frac{t}{(t+\tau)^{3/2}}W_J(\tau) V_0\bigg(\frac{3}{2}, \frac{-q(t+\tau)}{2\pi Nr}\bigg).
\end{equation*}
Since $|\tau|\ll t^{1-\epsilon}$,
\begin{equation*}
\frac{\partial}{\partial \tau} W_J(q,r,\tau)\ll \frac{1}{t^{1/2}|\tau|}.
\end{equation*}

When $n_2=0$, the bounds on $U^\dagger(0, i(\tau_2-\tau_1))$ gives arbitrary saving if $|\tau_2-\tau_1|\gg t^{\epsilon}$. Recalling $a_i\sim Q$, $|r_i|\ll Ct^{1+\epsilon}/N$ and $|\tau_i|\ll Nt^\epsilon/QC$, the bound on $\frK$ in this case is
\begin{equation*}
\frK\ll \frac{QC}{Nt}t^\epsilon.
\end{equation*}

Next, we estimate $\frK$ when $n_2\neq0$. Applying \cite[Lemma 5]{Mun4}
\begin{equation*}
\begin{split}
U^\dagger\bigg(\frac{n_2L}{q_1q_2}, i(\tau_2-\tau_1)\bigg) = \frac{c_3}{(\tau_2-\tau_1)^{1/2}}&U\bigg(\frac{(\tau_2-\tau_1)q_1q_2}{2\pi n_2L}\bigg)\bigg(\frac{(\tau_2-\tau_1)q_1q_2}{2\pi en_2L}\bigg)^{i(\tau_2-\tau_1)} \\
&+ O\bigg(\min\bigg\{\frac{1}{|\tau_2-\tau_1|^{3/2}}, \frac{C^3}{(|n_2|L)^{3/2}}\bigg\}\bigg),
\end{split}
\end{equation*}
for some constant $c_3$ (which depends on the signs of $n_2$ and $(\tau_2-\tau_1)$). The contribution of this error term towards $\frK$ is
\begin{equation*}
O\bigg(t^\epsilon\frac{Q^2C^2}{N^2t}\underset{[J, 4J/3]}{\int\int}\min\bigg\{\frac{1}{|\tau_2-\tau_1|^{3/2}}, \frac{C^3}{(|n_2|L)^{3/2}}\bigg\}d\tau_1 d\tau_2\bigg).
\end{equation*}
Since $|\tau_i|\sim J\ll  Nt^\epsilon/QC$, the above is bounded by
\begin{equation*}
t^\epsilon\frac{Q^2C^2}{N^2t}\frac{C}{(|n_2|L)^{1/2}}J\ll \frac{QC^2}{Nt}\frac{t^\epsilon}{(|n_2|L)^{1/2}}.
\end{equation*}

We therefore set
\begin{equation}\label{Bstar}
B^\star(C,0) = \frac{QCt^\epsilon}{Nt}, \quad \text{ and }\quad B^\star(C,n_2) = \frac{QC^2t^\epsilon}{Nt}\frac{1}{(|n_2|L)^{1/2}} \quad (\text{for } n_2\neq0).
\end{equation}

Finally we analyze the main term contribution towards $\frK$ and proceed exactly as Munshi \cite{Mun4} does. By Fourier inversion
\begin{equation*}
\bigg(\frac{(\tau_2-\tau_1)q_1q_2}{2\pi n_2L}\bigg)^{-1/2}U\bigg(\frac{(\tau_2-\tau_1)q_1q_2}{2\pi n_2L}\bigg) = \int_\BR U^\dagger(y,1/2)e\bigg(\frac{(\tau_2-\tau_1)q_1q_2}{2\pi n_2L}y\bigg)\ dy. 
\end{equation*}
Using the above and the Stirling approximation \eqref{stirling gl3} for $\gamma_+(s)$
\begin{equation}\label{final integral}
\frK = c_4\frac{r_1r_2a_1a_2}{t^2}\bigg(\frac{q_1q_2}{|n_2|L}\bigg)^{1/2}\int_\BR U^\dagger(y,1/2)\underset{\BR^2}{\int\int} g(\tau_1,\tau_2) e(f(\tau_1,\tau_2))\ d\tau_1\ d\tau_2\ dy + O(B^\star(C,n_2)),
\end{equation}
where
\begin{equation*}
\begin{split}
2\pi f(\tau_1,\tau_2) = & 3\tau_1\log(|\tau_1|/e\pi) - 3\tau_2\log(|\tau_2|/e\pi) - (\tau_1-\tau_2)\log LN + 3\tau_1\log q_1 - 3\tau_2\log q_2 \\
&- (t+\tau_1)\log\bigg(\frac{-(t+\tau_1)q_1}{2\pi eNr_1}\bigg) + (t+\tau_2)\log\bigg(\frac{-(t+\tau_2)q_2}{2\pi eNr_2}\bigg)\\
& - (\tau_1-\tau_2)\log\bigg(\frac{(\tau_2-\tau_1)q_1q_2}{2\pi en_2L}\bigg) + \frac{(\tau_2-\tau_1)q_1q_2}{n_2L}y,
\end{split}
\end{equation*}
and
\begin{equation*}
g(\tau_1,\tau_2) = \Phi_+(\tau_1)\overline{\Phi_+(\tau_2)}W_J(q_1,r_1,\tau_1)W_J(q_2,r_2,\tau_2).
\end{equation*}
Then
\begin{equation*}
\begin{split}
&2\pi\frac{\partial^2}{\partial^2\tau_1}f(\tau_1,\tau_2) = \frac{3}{\tau_1} - \frac{1}{t+\tau_1} + \frac{1}{\tau_2-\tau_1}, \qquad 2\pi\frac{\partial^2}{\partial^2\tau_2}f(\tau_1,\tau_2) = \frac{-3}{\tau_2} + \frac{1}{t+\tau_2} + \frac{1}{\tau_2-\tau_1},\\
&2\pi\frac{\partial^2}{\partial\tau_1\partial\tau_2}f(\tau_1,\tau_2) = \frac{1}{\tau_1-\tau_2}.
\end{split}
\end{equation*}
Therefore
\begin{equation*}
4\pi^2\bigg[\frac{\partial^2}{\partial^2\tau_1}f(\tau_1,\tau_2)\frac{\partial^2}{\partial^2\tau_2}f(\tau_1,\tau_2) - \bigg(\frac{\partial^2}{\partial\tau_1\partial\tau_2}f(\tau_1,\tau_2)\bigg)^2\bigg] = \frac{-6}{\tau_1\tau_2} + O(t^\epsilon/Jt).
\end{equation*}

We notice that $\partial^2 f/\partial^2\tau_1=0$ for $\tau_2 = (2t-3\tau_1)\tau_1/(3t-4\tau_1)$, and is therefore small when $\tau_1=(2/3+o(1))\tau_2$. We however recall that $\tau_i\in[J,4J/3]$ (since $W_J$ is supported there) and $[J,4J/3]\cap[2J/3, 8J/9]=\emptyset$. Therefore $\partial^2 f/\partial^2\tau_1\gg 1/|\tau_i|$. The same argument justifies why $\partial^2 f/\partial^2\tau_2\gg 1/|\tau_i|$. Therefore the conditions of Lemma \ref{double expo sum} hold for above with $p_1=p_2=1/J^{1/2}$. Since $\Phi_+'(\tau)\ll |\tau|\-$ and $W'_J(q,r,\tau)\ll t^{-1/2}|\tau|\-$ (the derivative with respect to $\tau$), the total variation of $g(\tau_1, \tau_2)$ is bounded as $var(g)\ll t^{-1+\epsilon}$. By applying Lemma \ref{double expo sum}, the $\tau_1,\tau_2$-integral is bounded by $O(Jt^{-1+\epsilon})$. Therefore the contribution of the leading term of \eqref{final integral} towards $\frK$ is bounded by
\begin{equation*}
O\bigg(\frac{Q^2C^2}{N^2}\frac{C}{(|n_2|L)^{1/2}}\frac{Nt^\epsilon}{QCt}\bigg) = O(B^\star(C,n_2)),
\end{equation*}
where we have used $J\ll Nt^\epsilon/QC$. After this analysis, we bound the expression of $S_{1,+,J}(N,C,L)$ as given in \eqref{S1jNCL}.

\subsection{Diagonal contribution} We first consider the contribution of $n_2=0$, which we denote by\\ $S_{1,+,j}^\flat(N,C,L)$. Using the second statement of Lemma \ref{character sum lemma} and the bound of $B^\star(C,0)$ as given in \eqref{Bstar}
\begin{equation*}
S_{1,+,J}^\flat(N,C,L)\ll \frac{1}{Q^2C^5}\frac{QCt^\epsilon}{Nt}\sum_{n_1\leq 2C} \bigg\lbrace \frac{C^5t}{n_1^2N} + \frac{C^5t^2}{n_1N^2}\bigg\rbrace,
\end{equation*}
where the first term is the contribution from terms with $r_1=r_2$ and the second term is the contribution from $r_1\neq r_2$. Since we will choose $N>t$, the first term dominates and we get
\begin{equation}\label{main diagonal}
S_{1,+,j}^\flat(N,C,L)\ll \frac{Ct^{\epsilon}}{N^2Q}.
\end{equation}

\begin{Remark}\label{main_remark}
The diagonal contribution as given in \eqref{main diagonal} improves over Munshi's corresponding estimate in \cite[Section 6.2]{Mun4}. Munshi estimates,
$
S_{1,+,j}^\flat(N,C,L)\ll {t^{\epsilon}}/{N^{3/2}K^{1/2}C}
$ 
where $K=N/Q^2$. Therefore,
$
S_{1,+,j}^\flat(N,C,L)\ll {Qt^{\epsilon}}/{N^{2}C}
$
with $N/t^{1-\epsilon}<C<Q$. Therefore Munshi obtains
\begin{equation*}
S_{1,+,j}^\flat(N,C,L)\ll \frac{Qt^{1+\epsilon}}{N^{3}}.
\end{equation*}
This bound is worse than the bound we obtain
\begin{equation*}
S_{1,+,j}^\flat(N,C,L)\ll \frac{Ct^{\epsilon}}{N^{2}Q} \ll \frac{t^{\epsilon}}{N^{2}}.
\end{equation*} 
This improvement helps us to improve upon the subconvex estimate of Munshi \cite{Mun4}.
\end{Remark}

\subsection{Off-diagonal contribution} We now bound the contribution of terms with $n_2\neq0$, which we denote by $S_{1,+,j}^\sharp(N,C,L)$. Using the first statement of Lemma \ref{character sum lemma} and the bound on $B^\star(C,n_2)$ as given in \eqref{Bstar}
\begin{equation}\label{main off-diagonal}
\begin{split}
S_{1,+,j}^\sharp(N,C,L)&\ll \frac{t^\epsilon}{Q^2C^3} \sum_{n_1\leq2C} \underset{\substack{C<q_1\leq 2C, (r_1,q_1)=1\\ 1\leq |r_1|\ll q_1t^{1+\epsilon}/N\\ n_1|q_1}}{\sum\sum}\  \underset{\substack{C<q_2\leq 2C, (r_2,q_2)=1\\ 1\leq |r_2|\ll q_2t^{1+\epsilon}/N\\ n_1|q_2}}{\sum\sum}\ n_1^2 \sum_{1\leq|n_2|\ll CNt^\epsilon/QL} (q,n_2)B^\star(C,n_2)\\
&\ll \frac{C^{7/2}t^{1+\epsilon}}{Q^{3/2}N^{5/2}L}\ll \frac{C^2t^{1+\epsilon}}{N^{5/2}L}.
\end{split}
\end{equation}

\subsection{Estimating $S_{1,J}(N,C)$} Using the bounds \eqref{main diagonal} and \eqref{main off-diagonal} in \eqref{main after cauchy}
\begin{equation*}
\begin{split}
S_{1,+,J}(N,C)\ll t^\epsilon N^{1/2} \sum_{\substack{1\leq L\ll N^2t^\epsilon/Q^3\\ dyadic}} \bigg( \frac{L^{1/2}C^{1/2}}{Q^{1/2}N} + \frac{Ct^{1/2}}{N^{5/4}}\bigg)\ll t^\epsilon N^{1/2}\bigg(\frac{C^{1/2}}{Q^2}+\frac{Ct^{1/2}}{N^{5/4}}\bigg).
\end{split}    
\end{equation*}

Multiplying by $N$ and summing dyadically over $J\in \CJ$ and $1\leq C\leq Q$, the contribution of the above expression towards $S^+(N)$ is
\begin{equation}\label{S1}
\frac{N^{3/2}t^\epsilon}{Q^{3/2}} + QN^{1/4}t^{1/2+\epsilon}.
\end{equation}

We notice that the respective diagonal and off-diagonal estimates in the main term \eqref{S1} are equal to or bigger than those of the error term \eqref{S2} if $Q<N^{1/2}$. Under this assumption 
\begin{equation*}
S(N)\ll \frac{N^{3/2}t^\epsilon}{Q^{3/2}} + QN^{1/4}t^{1/2+\epsilon}.
\end{equation*}
The optimum choice of $Q$ is therefore $Q=N^{1/2}/t^{1/5}$. Thus the condition $Q<N^{1/2}$ is satisfied. Finally, we observe that the condition $N/t<Q$ is the same as $N/t<N^{1/2}/t^{1/5}$, which is always satisfied since $N<t^{3/2+\epsilon}$. Thus we get the final bound of $S(N)\ll N^{3/4}t^{3/10+\epsilon}$, which proves Proposition \ref{main prop gl3}.

\subsection*{Acknowledgements} The author would like to thank Prof. Roman Holowinsky for his encouragement and suggestions that improved the presentation of this paper.

\bibliographystyle{abbrv}
\bibliography{ref}

\end{document}